\documentclass[12pt,reqno]{amsart}
\usepackage{amsfonts, amssymb, latexsym, graphics, color}

\usepackage{stmaryrd}		

\usepackage{algorithmic}
\usepackage{algorithm}

\theoremstyle{plain}
\newtheorem{theorem}{Theorem}[section]
\newtheorem{proposition}[theorem]{Proposition}

\newtheorem{lemma}[theorem]{Lemma}

\theoremstyle{definition}
\newtheorem{definition}[theorem]{Definition}

\newtheorem{example}[theorem]{Example}

\numberwithin{equation}{section}

\usepackage{tikz}
\usetikzlibrary{matrix,arrows}
\usetikzlibrary{positioning}
\usetikzlibrary{fit}
\usetikzlibrary{patterns}
\usetikzlibrary{arrows}

\newcommand{\id}{\mathrm{id}}

\DeclareMathOperator{\Av}{\mathrm{Av}}

\DeclareMathOperator{\unSvar}{{\mathsf{cand}}}
\DeclareMathOperator{\inv}{{\mathsf{inv}}}
\DeclareMathOperator{\ninv}{{\mathsf{ninv}}}
\DeclareMathOperator{\symS}{{\mathcal{S}}}

\newcommand{\sepp}{\,|\,} 							

\newcommand{\dbrac}[1]{{\llbracket #1 \rrbracket}} 	

\newcommand{\pattern}[4]{										
  \raisebox{0.6ex}{
  \begin{tikzpicture}[scale=0.35, baseline=(current bounding box.center), #1]
  \useasboundingbox (0.0,-0.1) rectangle (#2+1.4,#2+1.1);
    \foreach \x/\y in {#4}
      \fill[pattern color = black!65, pattern=north east lines] (\x,\y) rectangle +(1,1);
    \draw (0.01,0.01) grid (#2+0.99,#2+0.99);
    \foreach \x/\y in {#3}
      \filldraw (\x,\y) circle (6pt);
  \end{tikzpicture}}
}

\newcommand{\patternnl}[4]{										
  \raisebox{0.6ex}{
  \begin{tikzpicture}[scale=0.35, baseline=(current bounding box.center), #1]
  \useasboundingbox (0.85,-0.1) rectangle (#2+1.4,#2+1.1);
    \foreach \x/\y in {#4}
      \fill[pattern color = black!65, pattern=north east lines] (\x,\y) rectangle +(1,1);
    \foreach \x/\y in {#3}
      \filldraw (\x,\y) circle (6pt);
  \end{tikzpicture}}
}

\newcommand{\onetwo}{\patternnl{scale=0.2}{2}{1/1,2/2}{}}		

\newcommand{\imopattern}[6]{									
  \raisebox{0.6ex}{
  \begin{tikzpicture}[scale=0.35, baseline=(current bounding box.center), #1]
  \useasboundingbox (0.0,-0.1) rectangle (#2+1.4,#2+1.1);
    \foreach \x/\y in {#6}
      \fill[pattern color = black!65, pattern=north east lines] (\x,\y) rectangle +(1,1);
    \draw (0.01,0.01) grid (#2+0.99,#2+0.99);
    \foreach \x/\y in {#4}
      \draw[fill=white] (\x,\y) circle (6pt);
    \foreach \x/\y in {#5}
      \draw[fill=white] (\x,\y) circle (11pt);
    \foreach \x/\y in {#3}
      \filldraw (\x,\y) circle (6pt);
  \end{tikzpicture}}
}

\newcommand{\patternsbm}[5]{									
  \raisebox{0.6ex}{
  \begin{tikzpicture}[scale=0.35, baseline=(current bounding box.center), #1]
  \useasboundingbox (0.0,-0.1) rectangle (#2+1.4,#2+1.1);
    \foreach \x/\y in {#4}
      \fill[pattern color = black!65, pattern=north east lines] (\x,\y) rectangle +(1,1);
    \draw (0.01,0.01) grid (#2+0.99,#2+0.99);
    \foreach \x/\y/\z/\w/\A in {#5}
       \fill[color = white!100, opacity=1, rounded corners = 1.5pt] (\x+0.125,\y+0.125) rectangle (\z-0.125,\w-0.125);
    \foreach \x/\y/\z/\w/\A in {#5}
       \draw[color = black, rounded corners = 1.5pt] (\x+0.125,\y+0.125) rectangle (\z-0.125,\w-0.125);
    \foreach \x/\y/\z/\w/\A in {#5}
       \fill[black] (\x/2+\z/2,\y/2+\w/2) node {$\scriptstyle\A$};
    \foreach \x/\y in {#3}
      \filldraw (\x,\y) circle (6pt);

  \end{tikzpicture}}
}

\newcommand{\imopatternsbm}[7]{									
  \raisebox{0.6ex}{
  \begin{tikzpicture}[scale=0.35, baseline=(current bounding box.center), #1]
  \useasboundingbox (0.0,-0.1) rectangle (#2+1.4,#2+1.1);
    \foreach \x/\y in {#6}
      \fill[pattern color = black!65, pattern=north east lines] (\x,\y) rectangle +(1,1);
    \draw (0.01,0.01) grid (#2+0.99,#2+0.99);
    \foreach \x/\y/\z/\w/\A in {#7}
       \fill[color = white!100, opacity=1, rounded corners = 1.5pt] (\x+0.125,\y+0.125) rectangle (\z-0.125,\w-0.125);
    \foreach \x/\y/\z/\w/\A in {#7}
       \draw[color = black, rounded corners = 1.5pt] (\x+0.125,\y+0.125) rectangle (\z-0.125,\w-0.125);
    \foreach \x/\y/\z/\w/\A in {#7}
       \fill[black] (\x/2+\z/2,\y/2+\w/2) node {$\scriptstyle\A$};
    \foreach \x/\y in {#4}
      \draw[fill=white] (\x,\y) circle (6pt);
    \foreach \x/\y in {#5}
      \draw[fill=white] (\x,\y) circle (11pt);
    \foreach \x/\y in {#3}
      \filldraw (\x,\y) circle (6pt);
  \end{tikzpicture}}
}

\newcommand{\decpattern}[6]{									
  \raisebox{0.6ex}{
  \begin{tikzpicture}[scale=0.35, baseline=(current bounding box.center), #1]
  \useasboundingbox (0.0,-0.1) rectangle (#2+1.4,#2+1.1);
    \foreach \x/\y in {#4}
      {
      \fill[pattern color = black!65, pattern=north east lines] (\x,\y) rectangle +(1,1);
      }
    \draw (0.01,0.01) grid (#2+0.99,#2+0.99);
       
    \foreach \x/\y/\z/\w/\A in {#5}
       {
       \fill[color = white!100, opacity=1, rounded corners=1.5pt] (\x+0.125,\y+0.125) rectangle (\z-0.125,\w-0.125);
       \draw[color = black, rounded corners=1.5pt] (\x+0.125,\y+0.125) rectangle (\z-0.125,\w-0.125);
       \fill[black] (\x/2+\z/2,\y/2+\w/2) node {$\scriptstyle\A$};
       }
    \foreach \x/\y/\z/\w/\A in {#6}
       {
       \fill[color = white!100, opacity=1, rounded corners=1.5pt] (\x+0.125,\y+0.125) rectangle (\z-0.125,\w-0.125);
       \fill[pattern color = gray, pattern=north east lines, rounded corners=1.5pt] (\x+0.125,\y+0.125) rectangle (\z-0.125,\w-0.125);
       \draw[color = black, rounded corners=1.5pt] (\x+0.125,\y+0.125) rectangle (\z-0.125,\w-0.125);
       \fill[black] (\x/2+\z/2,\y/2+\w/2) node {$\scriptstyle\A$};
       }
    \foreach \x/\y in {#3}
      \filldraw (\x,\y) circle (6pt);

  \end{tikzpicture}}
}

\newcommand{\decpatternww}[8]{								
  \raisebox{0.6ex}{
  \begin{tikzpicture}[scale=0.35, baseline=(current bounding box.center), #1]
  \useasboundingbox (0.0,-0.1) rectangle (#2+1.4,#2+1.1);
    \foreach \x/\y in {#6}
      {
      \fill[pattern color = black!65, pattern=north east lines] (\x,\y) rectangle +(1,1);
      }
    \draw (0.01,0.01) grid (#2+0.99,#2+0.99);
       
    \foreach \x/\y/\z/\w/\A in {#7}
       {
       \fill[color = white!100, opacity=1, rounded corners=1.5pt] (\x+0.125,\y+0.125) rectangle (\z-0.125,\w-0.125);
       \draw[color = black, rounded corners=1.5pt] (\x+0.125,\y+0.125) rectangle (\z-0.125,\w-0.125);
       \fill[black] (\x/2+\z/2,\y/2+\w/2) node {$\scriptstyle\A$};
       }
    \foreach \x/\y/\z/\w/\A in {#8}
       {
       \fill[color = white!100, opacity=1, rounded corners=1.5pt] (\x+0.125,\y+0.125) rectangle (\z-0.125,\w-0.125);
       \fill[pattern color = gray, pattern=north east lines, rounded corners=1.5pt] (\x+0.125,\y+0.125) rectangle (\z-0.125,\w-0.125);
       \draw[color = black, rounded corners=1.5pt] (\x+0.125,\y+0.125) rectangle (\z-0.125,\w-0.125);
       \fill[black] (\x/2+\z/2,\y/2+\w/2) node {$\scriptstyle\A$};
       }
    \foreach \x/\y in {#4}
      \draw[fill=white] (\x,\y) circle (6pt);
    \foreach \x/\y in {#5}
      \draw[fill=white] (\x,\y) circle (11pt);
    \foreach \x/\y in {#3}
      \filldraw (\x,\y) circle (6pt);

  \end{tikzpicture}}
}

\newcommand{\vinc}[3]{
\begin{tikzpicture}[baseline, inner sep = 0mm]

	\begin{scope}[yshift = 3]
	
	\foreach \x/\y in {#2}
	{
		\node at (\x*0.2,-0.14)  [label=$\y$] {};  
	}
	
	\foreach \z in {#3}
	{
		\ifnum 0<\z
			\ifnum \z<#1
				\draw[thick] (\z*0.2-0.07,-0.17) -- (\z*0.2+0.27,-0.17);
			\fi
		\fi
		
		\ifnum 0=\z
			\draw[thick] (0.08,0.1) -- (0.08,-0.17) -- (0.26,-0.17);
		\fi
		
		\ifnum \z=#1
			\draw[thick] (\z*0.2+0.13,0.1) -- (\z*0.2+0.13,-0.17) -- (\z*0.2-0.05,-0.17);
		\fi
	}
	\end{scope}
\end{tikzpicture}
}

\newcommand{\vinci}[3]{
\begin{tikzpicture}[baseline, scale = 0.8, , inner sep = 0mm]

	\begin{scope}[yshift = 3]
	
	\foreach \x/\y in {#2}
	{
		\node at (\x*0.2,-0.14)  [label=$\scriptstyle{\y}$] {};  
	}
	
	\foreach \z in {#3}
	{
		\ifnum 0<\z
			\ifnum \z<#1
				\draw[thick] (\z*0.2-0.07,-0.17) -- (\z*0.2+0.27,-0.17);
			\fi
		\fi
		
		\ifnum 0=\z
			\draw[thick] (0.08,0.1) -- (0.08,-0.17) -- (0.26,-0.17);
		\fi
		
		\ifnum \z=#1
			\draw[thick] (\z*0.2+0.13,0.1) -- (\z*0.2+0.13,-0.17) -- (\z*0.2-0.05,-0.17);
		\fi
	}
	\end{scope}
\end{tikzpicture}
}

\newcommand{\bivinc}[4]{
\begin{tikzpicture}[baseline, inner sep = 0mm]

	\begin{scope}[yshift = -2]
	
	\foreach \x/\y in {#2}
	{
		\node at (\x*0.2,0.15)   [label=$\x$] {};  
		\node at (\x*0.2,-0.14)  [label=$\y$] {};  
	}
	
	\foreach \z in {#3}
	{
		\ifnum 0<\z
			\ifnum \z<#1
				\draw[thick] (\z*0.2-0.07,-0.17) -- (\z*0.2+0.27,-0.17);
			\fi
		\fi
		
		\ifnum 0=\z
			\draw[thick] (0.08,0.1) -- (0.08,-0.17) -- (0.21,-0.17);
		\fi
		
		\ifnum \z=#1
			\draw[thick] (\z*0.2+0.13,0.1) -- (\z*0.2+0.13,-0.17) -- (\z*0.2,-0.17);
		\fi
	}
	
	\foreach \z in {#4}
	{
		\ifnum 0<\z
			\ifnum \z<#1
				\draw[thick] (\z*0.2-0.07,0.47) -- (\z*0.2+0.27,0.47);
			\fi
		\fi
		
		\ifnum 0=\z
			\draw[thick] (0.08,0.21) -- (0.08,0.47) -- (0.21,0.47);
		\fi
		
		\ifnum \z=#1
			\draw[thick] (\z*0.2+0.13,0.21) -- (\z*0.2+0.13,0.47) -- (\z*0.2,0.47);
		\fi
	}
	\end{scope}
\end{tikzpicture}
}

\newcommand{\bivincs}[5]{
\begin{tikzpicture}[baseline, inner sep = 0mm]

	\begin{scope}[yshift = -2]

	\foreach \x/\w in {#2}
	{
		\node at (\x*0.2,-0.14)   [label=$\w$] {}; 
	}
	
	\foreach \x/\w in {#3}
	{
		\node at (\x*0.2,0.15) [label=$\w$] {}; 
	}
	
	\foreach \z in {#4}
	{
		\ifnum 0<\z
			\ifnum \z<#1
				\draw[thick] (\z*0.2-0.07,-0.17) -- (\z*0.2+0.27,-0.17);
			\fi
		\fi
		
		\ifnum 0=\z
			\draw[thick] (0.08,0.1) -- (0.08,-0.17) -- (0.26,-0.17);
		\fi
		
		\ifnum \z=#1
			\draw[thick] (\z*0.2+0.13,0.1) -- (\z*0.2+0.13,-0.17) -- (\z*0.2-0.05,-0.17);
		\fi
	}
	
	\foreach \z in {#5}
	{
		\ifnum 0<\z
			\ifnum \z<#1
				\draw[thick] (\z*0.2-0.07,0.47) -- (\z*0.2+0.27,0.47);
			\fi
		\fi
		
		\ifnum 0=\z
			\draw[thick] (0.08,0.21) -- (0.08,0.47) -- (0.26,0.47);
		\fi
		
		\ifnum \z=#1
			\draw[thick] (\z*0.2+0.13,0.21) -- (\z*0.2+0.13,0.47) -- (\z*0.2-0.05,0.47);
		\fi
	}
	\end{scope}
\end{tikzpicture}
}

\newcommand{\bivinci}[4]{
\begin{tikzpicture}[baseline, scale = 0.8, inner sep = 0mm]

	\begin{scope}[yshift = -2]
	
	\foreach \x/\y in {#2}
	{
		\node at (\x*0.2,0.15)   [label=$\scriptstyle{\x}$] {};  
		\node at (\x*0.2,-0.14)  [label=$\scriptstyle{\y}$] {};  
	}
	
	\foreach \z in {#3}
	{
		\ifnum 0<\z
			\ifnum \z<#1
				\draw[thick] (\z*0.2-0.07,-0.17) -- (\z*0.2+0.27,-0.17);
			\fi
		\fi
		
		\ifnum 0=\z
			\draw[thick] (0.08,0.1) -- (0.08,-0.17) -- (0.26,-0.17);
		\fi
		
		\ifnum \z=#1
			\draw[thick] (\z*0.2+0.13,0.1) -- (\z*0.2+0.13,-0.17) -- (\z*0.2-0.05,-0.17);
		\fi
	}
	
	\foreach \z in {#4}
	{
		\ifnum 0<\z
			\ifnum \z<#1
				\draw[thick] (\z*0.2-0.07,0.47) -- (\z*0.2+0.27,0.47);
			\fi
		\fi
		
		\ifnum 0=\z
			\draw[thick] (0.08,0.21) -- (0.08,0.47) -- (0.26,0.47);
		\fi
		
		\ifnum \z=#1
			\draw[thick] (\z*0.2+0.13,0.21) -- (\z*0.2+0.13,0.47) -- (\z*0.2-0.05,0.47);
		\fi
	}
	\end{scope}
\end{tikzpicture}
}

\begin{document}

\title[Sorting and preimages of pattern classes]{Sorting and preimages of pattern classes}

\author[Claesson]{Anders Claesson}
\author[\'Ulfarsson]{Henning \'Ulfarsson}

\address[Claesson]{Department of Computer and Information Sciences, University of Strathclyde, Glasgow G1 1XH, UK}
\address[\'Ulfarsson]{School of Computer Science, Reykjav\'ik University, Menntavegi 1, 101 Reykjav\'ik, Iceland}

\email{anders.claesson@cis.strath.ac.uk, henningu@ru.is}


\begin{abstract}
We introduce an algorithm to determine when a sorting operation, such as stack-sort or bubble-sort, outputs a given pattern.
  The algorithm provides a new proof of the
  description of West-$2$-stack-sortable permutations, that is
  permutations that are completely sorted when passed twice through a
  stack, in terms of patterns. We also
  solve the long-standing problem of describing
  West-$3$-stack-sortable permutations. This requires a new type of
  generalized permutation pattern we call a decorated pattern.
\end{abstract}



\maketitle

\setcounter{tocdepth}{1}
\tableofcontents

\section{Introduction}
\label{sec:in}
The set of permutations of $\{1,\dotsc,n\}$ is denoted $\symS_n$.
Permutations will be written in one-line notation, and the identity permutation $12\dotsm n$ will be denoted $\id_n$, or just $\id$ if $n$ is understood from context.

In the 1970's Knuth~\cite{MR0378456} initiated the study of sorting
and pattern avoidance in permutations. He considered the problem of
sorting a permutation by passing it through a stack.  A \emph{stack}
is a last in, first out data structure with two fundamental
operations: the \emph{push} operation moves an item from the input to
the top of the stack; the \emph{pop} operation moves an item from the
top of the stack to the output. Consider trying to sort the
permutation $231$ by a stack, as shown in Figure \ref{fig:231stack}.
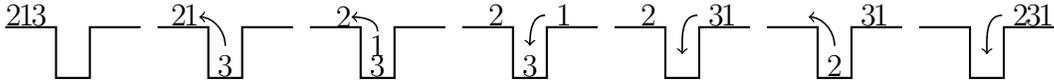
\begin{figure}[htbp]
\begin{center}

\begin{tikzpicture}
 [scale = 0.225, bend angle=37.5, pre/.style={<-,shorten <=1pt,semithick},
   post/.style={->,shorten >=1pt,semithick}]
  
 \begin{scope}[xshift=54cm]
 \draw [black,thick] (-4,-0.35) -- (-1,-0.35) -- (-1,-3.35) -- (1,-3.35) -- (1,-0.35) -- (4,-0.35);
 \node[] at (0,-2.75) (stack1) {};
 
 \node[] at (2,0.35) (first) {2}
 	edge [post,in=90, out=180]		(stack1);
 \node[] at (2.75,0.35) (second) {3};
 \node[] at (3.5,0.35) (third) {1};
 \end{scope}
 
 \begin{scope}[xshift=45cm]
 \draw [black,thick] (-4,-0.35) -- (-1,-0.35) -- (-1,-3.35) -- (1,-3.35) -- (1,-0.35) -- (4,-0.35);
 \node[] at (-2,0.45) (out1) {};
 
 \node[] at (2,0.35) (first) {3};
 \node[] at (2.75,0.35) (second) {1};
 
 \node[] at (0,-2.65) (stack1) {2};
 \node[] at (0,-2.1) {} edge [post, in=0, out=90] (-1.7,0.25);
 \end{scope}
 
 \begin{scope}[xshift=36cm]
 \draw [black,thick] (-4,-0.35) -- (-1,-0.35) -- (-1,-3.35) -- (1,-3.35) -- (1,-0.35) -- (4,-0.35);
 \node[] at (0,-2.75) (stack1) {};
 
 \node[] at (2,0.35) (first) {3}
 	edge [post,in=90, out=180]		(stack1);
 \node[] at (2.75,0.35) (second) {1};
 
 \node[] at (-2,0.35) (out1) {2};
 \end{scope}
 
 \begin{scope}[xshift=27cm]
 \draw [black,thick] (-4,-0.35) -- (-1,-0.35) -- (-1,-3.35) -- (1,-3.35) -- (1,-0.35) -- (4,-0.35);
 \node[] at (0,-2.25) (stack2) {};
  
 \node[] at (2,0.35) (first) {1}
 	edge [post,in=90, out=180]		(stack2);
 
 \node[] at (0,-2.65) (stack1) {3};
 
 \node[] at (-2,0.35) (out1) {2};
 \end{scope}
 
 \begin{scope}[xshift=18cm]
 \draw [black,thick] (-4,-0.35) -- (-1,-0.35) -- (-1,-3.35) -- (1,-3.35) -- (1,-0.35) -- (4,-0.35);
 \node[] at (-2,0.25) (out1) {2};
 
 \node[] at (0,-2.65) (stack1) {3};
 \node[] at (0,-1.45) (stack2) {1};
 \node[] at (0,-1.25) {} edge [post,in=0, out=90] (-1.7,0.25);
 \end{scope}
 
 \begin{scope}[xshift=9cm]
 \draw [black,thick] (-4,-0.35) -- (-1,-0.35) -- (-1,-3.35) -- (1,-3.35) -- (1,-0.35) -- (4,-0.35);
 \node[] at (-2,0.35) (out1) {1};
 
 \node[] at (0,-2.65) (stack1) {3};
 
 \node[] at (0,-2.1) {} edge [post, in=0, out=90] (-1.7,0.25);
 
 \node[] at (-2.75,0.35) (out2) {2};
 \end{scope}
 
 \begin{scope}[xshift=0cm]
 \draw [black,thick] (-4,-0.35) -- (-1,-0.35) -- (-1,-3.35) -- (1,-3.35) -- (1,-0.35) -- (4,-0.35);
 
 \node[] at (-2,0.35) (out1) {3};
 \node[] at (-2.75,0.35) (out2) {1};
 \node[] at (-3.5,0.35) (out3) {2};
 \end{scope}
  
\end{tikzpicture}

\caption{Trying, and failing, to sort $231$ with a stack. The figure is read from right to left}
\label{fig:231stack}
\end{center}
\end{figure}

Note that we always want the elements in the stack to be increasing,
from the top, since otherwise it would be impossible for the output to
be sorted.  We failed to sort the permutation in one pass and
therefore say that it is not \emph{stack-sortable}, which Knuth showed is
equivalent to containg $231$ as a
pattern.  We will reprove this in
Theorem~\ref{thm:Knuth}. Several variations on Knuth's original
problem have been considered, see B\'ona~\cite{MR2028290} for a
survey. In this paper we consider repeatedly passing a
permutation through a stack, while keeping the elements on the stack in increasing
order from top to bottom. We also consider the bubble-sort
operator. We introduce a new method for finding patterns in a
permutation that will cause these sorting devices to output a given
pattern, and use this to solve the long-standing problem of describing
West-$3$-stack-sortable permutations. If the given pattern is classical (defined below) we show that the mesh
patterns introduced by Br\"and\'en and the first author~\cite{BC11} suffice,
but if the given pattern is itself a mesh pattern we will need to
introduce a new kind of generalized pattern we call a \emph{decorated}
pattern.

In Section~\ref{sec:algo} we describe an algorithm that given a classical pattern $p$ produces a finite list of (marked) mesh patterns $P$ such that $S^{-1}(\Av(p)) = \Av(P)$, where $S$ is the stack-sort operator. This algorithm automates proving some of the statements in the previous sections and can be extended to cover the bubble-sort operator as well.

We hope that our work is a step towards the general comprehension of the stack-sort operator and other similar operators. 
This extended abstract will appear at FPSAC 2012 and is based on two papers of the authors~\cite{U11W,CU11}.


\section{Generalized permutation patterns}
\label{sec:genpatts}

A \emph{standardization} of a list of numbers is another list of the same length such that the smallest letter in the original list has been replaced with $1$, the second smallest with $2$, \textit{etc}. The standardization of $5371$ is $3241$. A \emph{classical} (\emph{permutation}) \emph{pattern} is a permutation $p$ in $\symS_k$. A permutation $\pi$ in $\symS_n$ \emph{contains}, or \emph{has an occurrence of} the pattern $p$ if there are indices $1 \leq i_1 < \dotsb < i_k \leq n$ such that the standardization of $\pi(i_1) \dotsm \pi(i_k)$ equals the pattern $p$. If a permutation does not contain a pattern $p$ we say it \emph{avoids} $p$.
The permutation $\pi = 526413$ contains the pattern $p = 132$, and has three occurrences of it, in the subsequences $264$, $263$ and $243$. We can draw the \emph{graph} of the permutation by graphing the coordinates $(i,\pi(i))$ on a grid. For example the permutation $\pi$ above is shown in Figure \ref{fig:perm526413} where we have circled the occurrences of the pattern $p$.
\begin{figure}[htbp]
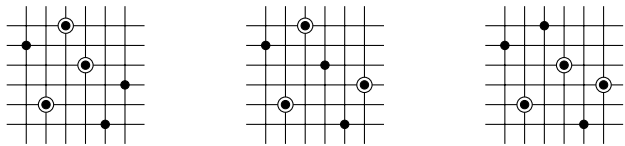

\begin{center}
\imopattern{scale=0.75}{ 6 }{ 1/5, 2/2, 3/6, 4/4, 5/1, 6/3 }
{}{ 2/2, 3/6, 4/4 }{}\qquad
\imopattern{scale=0.75}{ 6 }{ 1/5, 2/2, 3/6, 4/4, 5/1, 6/3 }
{}{ 2/2, 3/6, 6/3 }{}\qquad
\imopattern{scale=0.75}{ 6 }{ 1/5, 2/2, 3/6, 4/4, 5/1, 6/3 }
{}{ 2/2, 4/4, 6/3 }{}
\caption{The permutation $526413$ and three occurrences of the pattern $132$}
\label{fig:perm526413}
\end{center}
\end{figure}

\noindent
The same permutation avoids $123$, since we cannot find an increasing subsequence of length three.
%
%
%
%
%

\subsection{Mesh patterns and barred patterns}
Mesh patterns were introduced in \cite{BC11}. We review them via an example.
The mesh pattern
$
\pattern{scale=0.5}{ 3 }{ 1/1, 2/3, 3/2 }{0/2,1/2,2/2}
$
occurs in a permutation if we can find the underlying classical pattern $132$ positioned in such a way that the shaded regions are not occupied by other entries in the permutation. Another way of writing a mesh pattern is to give the underlying classical pattern, followed by the set of shaded boxes, labelled by their lower left corner (the left-most box in the bottom-most row being $(0,0)$). This mesh pattern is $(132,\{(0,2),(1,2),(2,2)\})$.

Consider the permutation $526413$. From above we know that the classical pattern has three occurrences in this permutation. In Figure \ref{fig:perm526413mesh} one can see that just one of these satisfies the additional requirement that there be no additional entries in the shaded region ``between and to the left of the $3$ and the $2$''.

\begin{figure}[htbp]
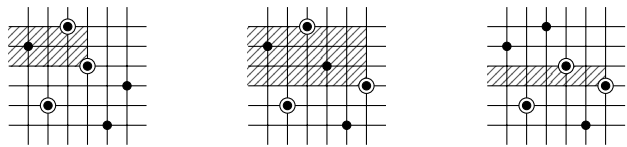

\begin{center}
\imopattern{scale=0.75}{ 6 }{ 1/5, 2/2, 3/6, 4/4, 5/1, 6/3 }
{}{ 2/2, 3/6, 4/4 }{0/4,0/5,1/4,1/5,2/4,2/5,3/4,3/5}\qquad
\imopattern{scale=0.75}{ 6 }{ 1/5, 2/2, 3/6, 4/4, 5/1, 6/3 }
{}{ 2/2, 3/6, 6/3 }{0/3,0/4,0/5,1/3,1/4,1/5,2/3,2/4,2/5,3/3,3/4,3/5,4/3,4/4,4/5,5/3,5/4,5/5}\qquad
\imopattern{scale=0.75}{ 6 }{ 1/5, 2/2, 3/6, 4/4, 5/1, 6/3 }
{}{ 2/2, 4/4, 6/3 }{0/3,1/3,2/3,3/3,4/3,5/3}
\caption{The permutation $526413$ and one occurrences of the mesh pattern $(132,\{(0,2),(1,2),(2,2)\})$}
\label{fig:perm526413mesh}
\end{center}
\end{figure}

\emph{Barred} patterns, introduced by West~\cite{W90}, are classical patterns with bars over some of the entries. Such a pattern is contained in a permutation if the standardization of the unbarred entries is contained in the permutation in such a way that they are not part of an occurrence of the whole pattern. The permutation $5264173$ contains the barred pattern $3\bar{5}241$, in the subsequence $5473$, since that is an occurrence of $3241$ that is not part of an occurrence of $35241$. In \cite{BC11} it was shown that any barred pattern with one barred entry is a mesh pattern. The barred pattern we discussed here is in fact the mesh pattern below.
$$
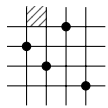
\pattern{scale=0.75}{ 4 }{ 1/3, 2/2, 3/4, 4/1 }{1/4}
$$

\subsection{Marked mesh patterns and decorated patterns}
Marked mesh patterns were introduced by the second author in \cite{U11} and give finer control over whether a certain region in a permutation is allowed to contain elements, and if so, how many. Again we just give an example.
Consider the marked mesh pattern
$
\patternsbm{scale=0.75}{ 3 }{ 1/1, 2/3, 3/2}{2/2}{1/0/3/2/\scriptstyle{1}}.
$
The meaning of the $1$ in the region containing boxes $(1,0)$, $(1,1)$, $(2,0)$ and $(2,1)$ is that this region must contain \emph{at least} one entry. In Figure \ref{fig:perm526413marked} we see that there is exactly one occurrence of this mesh pattern in the permutation $526413$.
\begin{figure}[htbp]
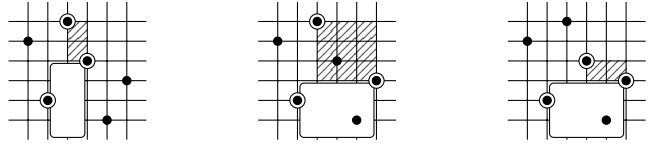

\begin{center}
\imopatternsbm{scale=0.75}{ 6 }{ 1/5, 2/2, 3/6, 4/4, 5/1, 6/3 }
{}{ 2/2, 3/6, 4/4 }{3/4,3/5}{2/0/4/4/{}} \qquad
\imopatternsbm{scale=0.75}{ 6 }{ 1/5, 2/2, 3/6, 4/4, 5/1, 6/3 }
{}{ 2/2, 3/6, 6/3 }{3/3,3/4,3/5,4/3,4/4,4/5,5/3,5/4,5/5}{2/0/6/3/{}} \qquad
\imopatternsbm{scale=0.75}{ 6 }{ 1/5, 2/2, 3/6, 4/4, 5/1, 6/3 }
{}{ 2/2, 4/4, 6/3 }{4/3,5/3}{2/0/6/3/{}}
\caption{The permutation $526413$ and one occurrence of a marked mesh pattern}
\label{fig:perm526413marked}
\end{center}
\end{figure}

Marked mesh patterns will be useful when we need to add elements into an existing pattern to ensure other elements are popped by a particular sorting device.
Below we will need even finer control over what is allowed inside a particular region in a pattern. We will need to control whether the entries in the region contain a particular pattern. Consider for example the decorated pattern
$
\decpattern{scale=1}{ 2 }{1/2, 2/1}{}{}{1/1/2/2/\onetwo}.
$
The decorated region in the middle signifies that an occurrence of this pattern should be an occurrence of the underlying classical pattern $21$ that additionally does not have entries in the region that contain the pattern $12$ -- or equivalently -- whatever is in that region must be in descending order, from left to right. In Figure \ref{fig:perm526413decorated} there is an occurrence of the decorated pattern on the left and on the right we have an occurrence of the classical pattern $21$ that does not satisfy the requirements of the decorated region.
Below we state the formal definition of a decorated pattern.
\begin{figure}[htbp]
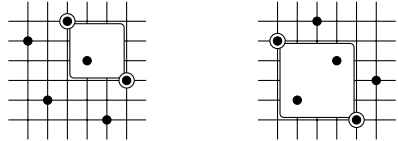

\begin{center}
\imopatternsbm{scale=0.75}{ 6 }{ 1/5, 2/2, 3/6, 4/4, 5/1, 6/3 }
{}{ 3/6, 6/3 }{}{3/3/6/6/{}} \qquad
\imopatternsbm{scale=0.75}{ 6 }{ 1/5, 2/2, 3/6, 4/4, 5/1, 6/3 }
{}{ 1/5, 5/1 }{}{1/1/5/5/{}}
\caption{The permutation $526413$ and one occurrences of a decorated pattern}
\label{fig:perm526413decorated}
\end{center}
\end{figure}

For integers $a < b$ we use $\dbrac{a,b}$ to denote the set $\{a, a+1, \dotsc, b\}$.

\begin{definition}
A \emph{decorated pattern} $(p,\mathcal{C})$ of length $k$ consists of a classical pattern
$p$ of length $k$ and a collection $\mathcal{C}$ which contains pairs $(C, q)$
where $C$ is a subset of the square $\dbrac{0,k} \times \dbrac{0,k}$ and $q$ is some pattern,
possibly another decorated pattern. An occurrence of $(p,\mathcal{C})$ in a permutation $\pi$
is a subset $\omega$
of the diagram $G(\pi) = \{ (i,\pi(i)) \sepp 1 \leq i \leq n \}$ such that there are order-preserving
injections $\alpha, \beta: \dbrac{1,k} \to \dbrac{1,n}$ satisfying two conditions:
\begin{enumerate}
\item $\omega = \{(\alpha(i),\beta(j)) \,\colon (i,j) \in G(p)\}$.
\item Let $R_{ij} = \dbrac{\alpha(i)+1, \alpha(i+1)-1} \times \dbrac{\beta(j)+1,  \beta(j+1) -1}$,
with $\alpha(0) = 0 = \beta(0)$ and $\alpha(k+1) = n+1 = \beta(k+1)$. For each pair $(C,q)$ we let
$C' = \bigcup_{(i,j) \in C} R_{ij}$ and require that
$
C' \cap G(\pi)
$
avoids $q$.
\end{enumerate}
\end{definition}

\section{Finding preimages of patterns}
\label{sec:preim}

In this section we define a method for describing patterns that are guaranteed to produce a given pattern in a permutation after it is sorted by a stack or with the bubble-sort operator.

\subsection{The stack-sort operator}
For a permutation $\pi$ we will denote with $S(\pi)$ the image of
$\pi$ after it is passed once through a stack. Note that a
permutation $\pi$ in $\symS_n$ is stack-sortable if and only if $S(\pi)
\in \Av_n(21)$, where $\Av_n(21)$ is the set of permutations of length
$n$ that avoid $21$. Of course $\Av_n(21) = \{\id\}$ but framing
the definition like this leads to a generalization: Given a pattern $p$, what
conditions need to be put on $\pi$ such that $S(\pi) \in \Av_n(p)$.
Exploring a sorting operator from this angle was first done by Albert
\textit{et al.}~\cite{AABCD10}; they, however, studied the bubble-sort
operator rather than the stack-sort operator.

Below we will call permutations such that $S^k(\pi) = \id$,
\emph{West-$k$-stack-sortable} permutations, since West considered
this generalization from the case of one stack first.
Note that for $k > 1$ these permutations are different from the \emph{$k$-stack-sortable} permutations,
which are the permutations that can be sorted by using $k$ stacks in series without the
requirement that the elements on the stack are increasing from top to bottom. For example,
the permutation $2341$ is not West-$2$-stack-sortable, but if we
put the entries $2,3,4$ onto the first stack, pass $1$ all the way to the end, and then use
the second stack to sort $2,3,4$ we end up with $1234$. So $2341$ is $2$-stack-sortable.

The basic idea
behind the method we are about to describe is that $S(\pi)$ has an
occurrence of a pattern $p$ of length $k$ if and only if the $k$
elements in this occurrence were present in $\pi$ as some kind of
pattern before we sorted. We start by showing how this idea allows us
to describe stack-sortable permutations as well as West-$2$-stack-sortable
permutations.

We know that $\pi$ is not sorted by the stack if and only if $S(\pi)$
contains the classical pattern $21$. Therefore consider a particular
occurrence of this pattern in $S(\pi)$. Before sorting, the elements
in this occurrence must have formed the pattern $ 21 =
\pattern{scale=0.5}{ 2 }{ 1/2, 2/1 }{} $ in $\pi$. In order to remain
in this order the element corresponding to $2$ must be popped off the
stack by a larger element before the element corresponding to $1$
enters. Thus the box $(1,2)$ must have at least one element and we
have an occurrence of the marked mesh pattern $
\patternsbm{scale=0.75}{ 2 }{ 1/2,
  2/1}{}{1/2/2/3/\scriptscriptstyle{1}} $ which is equivalent to the
pattern $ \pattern{scale=0.5}{ 3 }{ 1/2, 2/3, 3/1 }{}.  $ We have
therefore reproven Knuth's result.
\begin{theorem}[\cite{MR0378456}] \label{thm:Knuth}
  A permutation is stack-sortable if and only if it avoids $231$.
\end{theorem}

We can similarly reprove West's result on West-$2$-stack-sortable
permutations, \textit{i.e.}, permutations $\pi$ such that $S^2(\pi) =
\id$. By Knuth's result we know that $\pi$ will be sorted by two
passes through the stack if and only if $S(\pi)$ avoids the pattern
$231$. An occurrence of $231$ must have formed either of the patterns
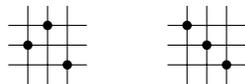
\begin{equation*} 
\pattern{scale=0.75}{ 3 }{ 1/2, 2/3, 3/1 }{} \qquad \pattern{scale=0.75}{ 3 }{ 1/3, 2/2, 3/1 }{}
\end{equation*}
in $\pi$.
For the elements in the pattern on the left
to stay in the same order as they pass through the stack we must have an element in
the box $(2,3)$ to pop the element corresponding to $3$ out
of the stack before the smallest element enters. Now, the opposite
happens for the pattern on the right. The $3$ must stay on the stack
until $2$ enters, so there can be no elements in the box $(1,3)$.
Then both $2$ and $3$ must leave the stack before $1$ enters. Thus the
patterns above become the two marked mesh patterns on the left
below. These are more naturally written as the two mesh patterns on
the right.
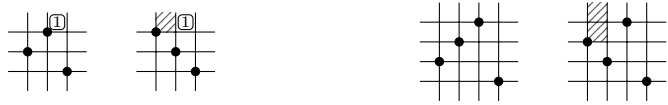
\[
\patternsbm{scale=0.75}{ 3 }{ 1/2, 2/3, 3/1 }{}{2/3/3/4/\scriptscriptstyle{1}} \quad \patternsbm{scale=0.75}{ 3 }{ 1/3, 2/2, 3/1 }{1/3}{2/3/3/4/\scriptscriptstyle{1}} \qquad\qquad\qquad \pattern{scale=0.75}{ 4 }{ 1/2, 2/3, 3/4, 4/1 }{} \quad \pattern{scale=0.75}{ 4 }{ 1/3, 2/2, 3/4, 4/1 }{1/3, 1/4}\vspace{-1ex}
\]
By a lemma of Hilmarsson, \textit{et al.}~\cite{HJSV11}, the last
pattern on the right is equivalent to $\pattern{scale=0.4, baseline=5pt}{ 4 }{ 1/3,
  2/2, 3/4, 4/1 }{1/4}$, in the sense that a permutation either contains both
patterns or avoids both. (For these two patterns it is also easy to
see directly that they are equivalent.) As mentioned above this
pattern is another representation of the barred pattern
$3\bar{5}241$. Thus we have re-derived West's result.

\begin{theorem}[\cite{W90}] \label{thm:West}
A permutation is West-$2$-stack-sortable if and only if it avoids $2341$ and $3\bar{5}241$.
\end{theorem}

\subsection{The bubble-sort operator}

The bubble-sort operator swaps adjacent entries in a permutation if
the left entry is larger than the right entry.
Let $B(\pi)$ denote the output of one pass of bubble-sort on $\pi$.
For instance, $B(521634) = 215346$. A modification of the method
above works equally well for $B$. We see that
$B(\pi)$ contains the pattern $21$ if and only if $\pi$ contains $ 21
= \pattern{scale=0.5}{ 2 }{ 1/2, 2/1 }{} $. To make sure that these
elements stay in this order, we either need a large element in front
of $2$, which would mean $2$ would never be moved; or
a large element in between $2$ and $1$ that will stop $2$
from moving past $1$.  We get the marked mesh pattern $
\patternsbm{scale=0.75}{ 2 }{ 1/2,
  2/1}{}{0/2/2/3/\scriptscriptstyle{1}}$, which is equivalent
to the two classical patterns $231$ and $321$. Thus,
$B(\pi) = \id$ if and only if $\pi$ avoids $231$ and $321$, as was first made explicit by Albert, \textit{et al.}~\cite{AABCD10}.
In the same paper the authors show that for any classical pattern $p$
with at least three left-to-right maxima, the third of which is not
the final symbol of $p$, the set $B^{-1}(\Av(p))$ is not described by
classical patterns. We consider the smallest example of such a
pattern, $p = 1243$. For a proof of the following proposition see \cite[Proposition~3.3]{U11W}.

\begin{proposition} \label{prop:bubble1243}
$\displaystyle{
B^{-1}(\Av(1243)) = \Av \left(
\patternsbm{scale=0.75}{ 4 }{ 1/1, 2/2, 3/4, 4/3 }{}{0/4/4/5/\scriptscriptstyle{1}},
\patternsbm{scale=0.75}{ 4 }{ 1/1, 2/4, 3/2, 4/3 }{0/4,1/4,2/4}{3/4/4/5/\scriptscriptstyle{1}},
\patternsbm{scale=0.75}{ 4 }{ 1/2, 2/1, 3/4, 4/3 }{0/2,0/3,0/4,1/2,1/3,1/4}{2/4/4/5/\scriptscriptstyle{1}},
\patternsbm{scale=0.75}{ 4 }{ 1/4, 2/1, 3/2, 4/3 }{0/4,1/4,2/4}{3/4/4/5/\scriptscriptstyle{1}}
\right).}
$
\end{proposition}
Note that all the patterns can be expanded to mesh patterns, but at the cost of having more patterns.


\subsection{West-$3$-stack-sortable permutations}

We now turn to West-$3$-stack-sortable permutations, \textit{i.e.}, permutations $\pi$ such that $S^3(\pi) = \id$. By West's result (Theorem~\ref{thm:West}) $S^3(\pi) = \id$ if and only if $S(\pi)$ avoids these two patterns:
\begin{equation} \label{eqWestpatts}
W_1 = \pattern{scale=0.75}{ 4 }{ 1/2, 2/3, 3/4, 4/1 }{} \qquad\quad W_2 = \pattern{scale=0.75}{ 4 }{ 1/3, 2/2, 3/4, 4/1 }{1/4}
\end{equation}
We will use the same method as we did above, but when we consider the pattern on the right, the shaded box will cause some complications and the decorated patterns introduced above will be necessary.

\begin{lemma} \label{lem:I's}
An occurrence of $W_1$ in $S(\pi)$ comes from exactly one of the patterns below in $\pi$.
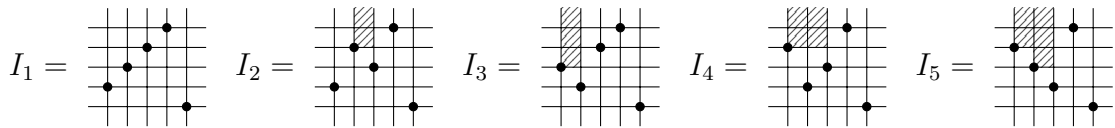
\begin{align*}
I_1 = \pattern{scale=0.75}{ 5 }{ 1/2, 2/3, 3/4, 4/5, 5/1 }{} \,\,\,\,
I_2 = \pattern{scale=0.75}{ 5 }{ 1/2, 2/4, 3/3, 4/5, 5/1 }{2/4,2/5} \,\,\,\,
I_3 = \pattern{scale=0.75}{ 5 }{ 1/3, 2/2, 3/4, 4/5, 5/1 }{1/3,1/4,1/5} \,\,\,\,
I_4 = \pattern{scale=0.75}{ 5 }{ 1/4, 2/2, 3/3, 4/5, 5/1 }{1/4,1/5,2/4,2/5} \,\,\,\,
I_5 = \pattern{scale=0.75}{ 5 }{ 1/4, 2/3, 3/2, 4/5, 5/1 }{1/4,1/5,2/3,2/4,2/5}
\end{align*}
\end{lemma}

\begin{proof}
Use Algorithm~\ref{algo:ShadeAndMark} from Section~\ref{sec:algo}. See also~\cite[Lemma~4.1]{U11W}.
\end{proof}

We now consider the pattern $W_2$, but without the shading.
\begin{lemma} \label{lem:W_2woShading}
An occurrence of $3241$
in $S(\pi)$ comes from exactly one of the patterns below in $\pi$.
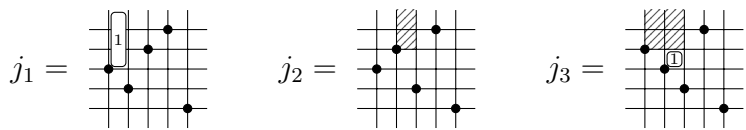
\begin{align*}
j_1 = \patternsbm{scale=0.75}{ 5 }{ 1/3, 2/2, 3/4, 4/5, 5/1 }{}{1/3/2/6/\scriptscriptstyle{1}} \qquad
j_2 = \pattern{scale=0.75}{ 5 }{ 1/3, 2/4, 3/2, 4/5, 5/1 }{2/4,2/5} \qquad
j_3 = \patternsbm{scale=0.75}{ 5 }{ 1/4, 2/3, 3/2, 4/5, 5/1 }{1/4,1/5,2/4,2/5}{2/3/3/4/\scriptscriptstyle{1}}
\end{align*}
\end{lemma}

\begin{proof}
Use Algorithm~\ref{algo:ShadeAndMark} from Section~\ref{sec:algo}. See also~\cite[Lemma~4.2]{U11W}.
\end{proof}

We now consider what additional conditions cause the patterns $j_1$, $j_2$ and $j_3$ in the lemma to give the correct shading in the pattern $W_2$. We express these conditions in the next lemma and two propositions.

\begin{lemma} \label{lem:j_3}
An occurrence of $j_3$ in a permutation $\pi$ will become an occurrence of $W_2$ in $S(\pi)$.
\end{lemma}
We leave the proof to the reader. We rename the pattern $J_3$ and note that it can also be expanded into a mesh pattern.
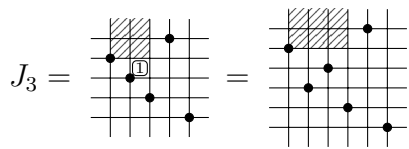
\[
J_3 = \patternsbm{scale=0.75}{ 5 }{ 1/4, 2/3, 3/2, 4/5, 5/1 }{1/4,1/5,2/4,2/5}{2/3/3/4/\scriptscriptstyle{1}}
= \pattern{scale=0.75}{ 6 }{ 1/5, 2/3, 3/4, 4/2, 5/6, 6/1 }{1/5,1/6,2/5,2/6,3/5,3/6}
\]

\begin{proposition} \label{prop:J2}
An occurrence of $j_2$ in a permutation $\pi$ becomes an occurrence of $W_2$ in $S(\pi)$ if and only
if it is part of one of the patterns below, where elements that have been added to $j_2$ are circled.
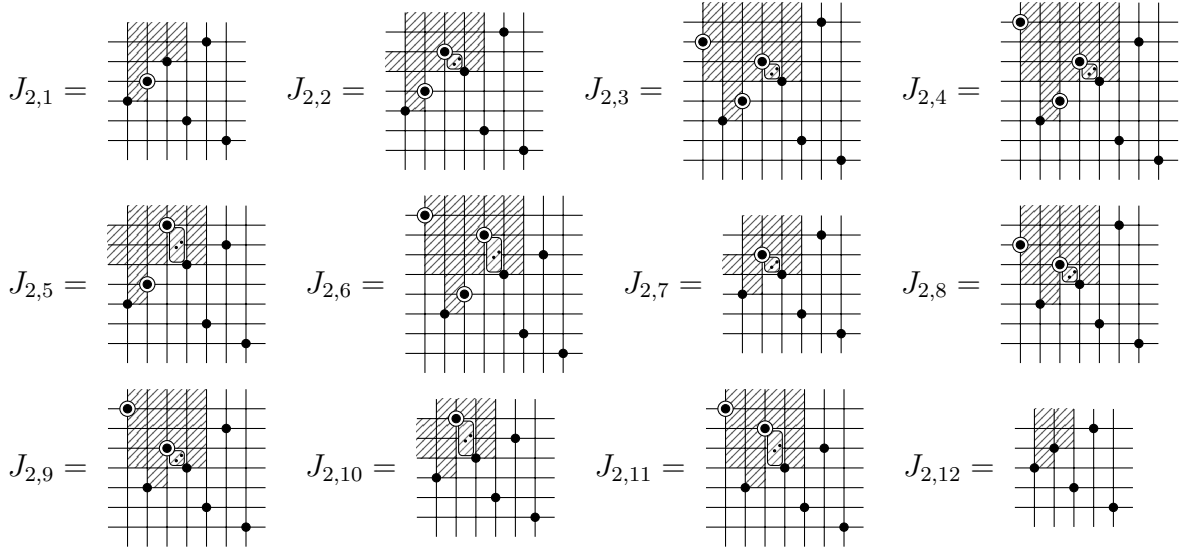
\begin{align*}
J_{2,1} &= \imopatternsbm{scale=0.75}{ 6 }{ 1/3, 2/4, 3/5, 4/2, 5/6, 6/1 }{}{2/4}{1/3,1/4,1/5,1/6,2/5,2/6,3/5,3/6}{} \quad
J_{2,2} = \decpatternww{scale=0.75}{ 7 }{ 1/3, 2/4, 3/6, 4/5, 5/2, 6/7, 7/1 }{}{2/4,3/6}{0/5,1/3,1/4,1/5,1/6,1/7,2/5,2/6,2/7,3/6,3/7,4/5,4/6,4/7}{}{3/5/4/6/{\onetwo}} \quad
J_{2,3} = \decpatternww{scale=0.75}{ 8 }{ 1/7, 2/3, 3/4, 4/6, 5/5, 6/2, 7/8, 8/1 }{}{1/7,3/4,4/6}{1/5,1/6,1/7,1/8,2/3,2/4,2/5,2/6,2/7,2/8,3/5,3/6,3/7,3/8,4/6,4/7,4/8,5/5,5/6,5/7,5/8}{}{4/5/5/6/{\onetwo}} \quad
J_{2,4} = \decpatternww{scale=0.75}{ 8 }{ 1/8, 2/3, 3/4, 4/6, 5/5, 6/2, 7/7, 8/1 }{}{1/8,3/4,4/6}{1/5,1/6,1/7,1/8,2/3,2/4,2/5,2/6,2/7,2/8,3/5,3/6,3/7,3/8,4/6,4/7,4/8,5/5,5/6,5/7,5/8}{}{4/5/5/6/{\onetwo}} \\
J_{2,5} &= \decpatternww{scale=0.75}{ 7 }{ 1/3, 2/4, 3/7, 4/5, 5/2, 6/6, 7/1 }{}{2/4,3/7}{0/5,0/6,1/3,1/4,1/5,1/6,1/7,2/5,2/6,2/7,3/7,4/5,4/6,4/7}{}{3/5/4/7/{\onetwo}} \quad
J_{2,6} = \decpatternww{scale=0.75}{ 8 }{ 1/8, 2/3, 3/4, 4/7, 5/5, 6/2, 7/6, 8/1 }{}{1/8,3/4,4/7}{1/5,1/6,1/7,1/8,2/3,2/4,2/5,2/6,2/7,2/8,3/5,3/6,3/7,3/8,4/7,4/8,5/5,5/6,5/7,5/8}{}{4/5/5/7/{\onetwo}} \quad
J_{2,7} = \decpatternww{scale=0.75}{ 6 }{ 1/3, 2/5, 3/4, 4/2, 5/6, 6/1 }{}{2/5}{0/4,1/3,1/4,1/5,1/6,2/5,2/6,3/4,3/5,3/6}{}{2/4/3/5/{\onetwo}} \quad
J_{2,8} = \decpatternww{scale=0.75}{ 7 }{ 1/6, 2/3, 3/5, 4/4, 5/2, 6/7, 7/1 }{}{1/6,3/5}{1/4,1/5,1/6,1/7,2/3,2/4,2/5,2/6,2/7,3/5,3/6,3/7,4/4,4/5,4/6,4/7}{}{3/4/4/5/{\onetwo}} \\
J_{2,9} &= \decpatternww{scale=0.75}{ 7 }{ 1/7, 2/3, 3/5, 4/4, 5/2, 6/6, 7/1 }{}{1/7,3/5}{1/4,1/5,1/6,1/7,2/3,2/4,2/5,2/6,2/7,3/5,3/6,3/7,4/4,4/5,4/6,4/7}{}{3/4/4/5/{\onetwo}} \quad
J_{2,10} = \decpatternww{scale=0.75}{ 6 }{ 1/3, 2/6, 3/4, 4/2, 5/5, 6/1 }{}{2/6}{0/4,0/5,1/3,1/4,1/5,1/6,2/6,3/4,3/5,3/6}{}{2/4/3/6/{\onetwo}} \quad
J_{2,11} = \decpatternww{scale=0.75}{ 7 }{ 1/7, 2/3, 3/6, 4/4, 5/2, 6/5, 7/1 }{}{1/7,3/6}{1/4,1/5,1/6,1/7,2/3,2/4,2/5,2/6,2/7,3/6,3/7,4/4,4/5,4/6,4/7}{}{3/4/4/6/{\onetwo}} \quad
J_{2,12} = \pattern{scale=0.75}{ 5 }{ 1/3, 2/4, 3/2, 4/5, 5/1 }{1/3,1/4,1/5,2/4,2/5}
\end{align*}
\end{proposition}

\begin{proof}
To ensure that there are no elements in the shaded box in $W_2$ we must look at the element that pops $3$ in $j_2$. There are four different possibilities.

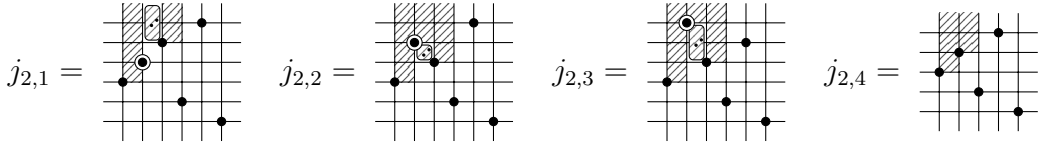
\begin{align*}
j_{2,1} &= \decpatternww{scale=0.75}{ 6 }{ 1/3, 2/4, 3/5, 4/2, 5/6, 6/1 }{}{2/4}{1/3,1/4,1/5,1/6,3/5,3/6}{}{2/5/3/7/{\onetwo}} \quad
j_{2,2} = \decpatternww{scale=0.75}{ 6 }{ 1/3, 2/5, 3/4, 4/2, 5/6, 6/1 }{}{2/5}{1/3,1/4,1/5,1/6,2/5,2/6,3/4,3/5,3/6}{}{2/4/3/5/{\onetwo}} \quad
j_{2,3} = \decpatternww{scale=0.75}{ 6 }{ 1/3, 2/6, 3/4, 4/2, 5/5, 6/1 }{}{2/6}{1/3,1/4,1/5,1/6,2/6,3/4,3/5,3/6}{}{2/4/3/6/{\onetwo}} \quad
j_{2,4} = \pattern{scale=0.75}{ 5 }{ 1/3, 2/4, 3/2, 4/5, 5/1 }{1/3,1/4,1/5,2/4,2/5}
\end{align*}
We explain the shadings and the decoration of the pattern $j_{2,2}$ as the others are similar. For this pattern, the size of the element that popped $3$ from the stack was in-between $4$ and $5$. Since this was the element that popped $3$ there can be no elements in boxes $(1,3), \dotsc, (1,6)$. The boxes $(2,5)$ and $(2,6)$ cannot contain an element, since that would pop the element we just added (the $5$ in $j_{2,2}$) and this element would land in the shaded box in $W_2$. Now consider the decorated box $(2,4)$. It can contain elements, but none of them are allowed to leave the stack prior to $4$ being pushed on, since any one of them would then land in the shaded box in $W_2$. Elements in this region must then be in descending order.

We must make sure that elements that arrived on the stack prior to $3$ are not popped into the shaded box in $W_2$. We only consider the pattern $j_{2,2}$ here.
If there are elements from box $(0,4)$ still on the stack when $3$ is put on they will be popped by $5$ and will land in the shaded box in $W_2$. We must therefore have this box empty, or an element in box $(0,5)$ or $(0,6)$ that pops everything before $3$ is pushed on. We get the three patterns $J_{2,7}, J_{2,8}, J_{2,9}$.
\end{proof}

We now consider what constraints must be imposed on the pattern $j_1$ to get $W_2$ after sorting.
\begin{proposition} \label{prop:J1}
An occurrence of $j_1$ in a permutation $\pi$ becomes an occurrence of $W_2$ in $S(\pi)$ if and only
if it is part of one of the patterns below, where elements that have been added to $j_1$ are circled.
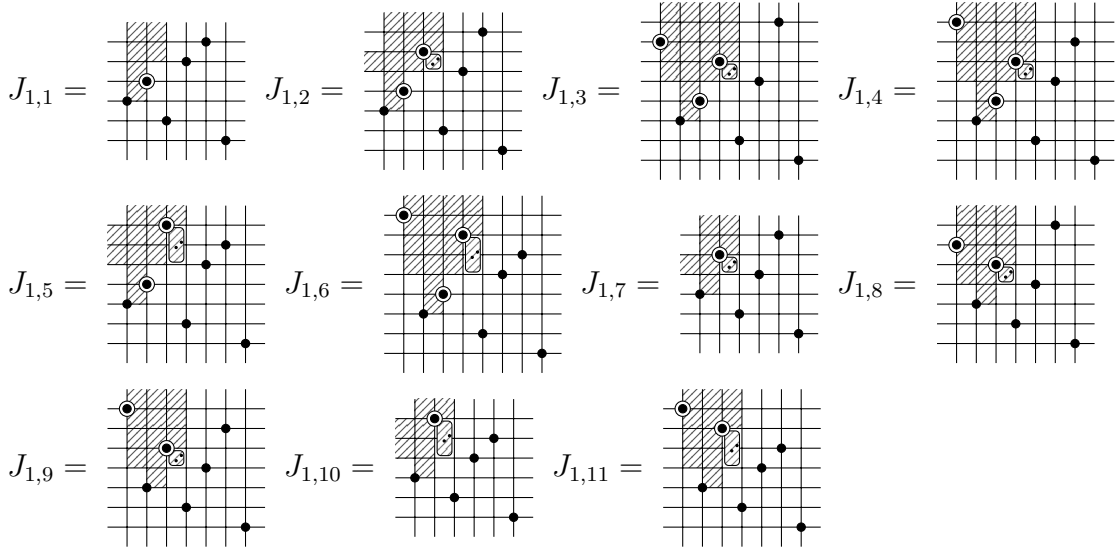
\begin{align*}
J_{1,1} &= \imopatternsbm{scale=0.75}{ 6 }{ 1/3, 2/4, 3/2, 4/5, 5/6, 6/1 }{}{2/4}{1/3,1/4,1/5,1/6,2/5,2/6}{} \,\,
J_{1,2} = \decpatternww{scale=0.75}{ 7 }{ 1/3, 2/4, 3/6, 4/2, 5/5, 6/7, 7/1 }{}{2/4,3/6}{0/5,1/3,1/4,1/5,1/6,1/7,2/5,2/6,2/7,3/6,3/7}{}{3/5/4/6/{\onetwo}} \,\,
J_{1,3} = \decpatternww{scale=0.75}{ 8 }{ 1/7, 2/3, 3/4, 4/6, 5/2, 6/5, 7/8, 8/1 }{}{1/7,3/4,4/6}{1/5,1/6,1/7,1/8,2/3,2/4,2/5,2/6,2/7,2/8,3/5,3/6,3/7,3/8,4/6,4/7,4/8}{}{4/5/5/6/{\onetwo}} \,\,
J_{1,4} = \decpatternww{scale=0.75}{ 8 }{ 1/8, 2/3, 3/4, 4/6, 5/2, 6/5, 7/7, 8/1 }{}{1/8,3/4,4/6}{1/5,1/6,1/7,1/8,2/3,2/4,2/5,2/6,2/7,2/8,3/5,3/6,3/7,3/8,4/6,4/7,4/8}{}{4/5/5/6/{\onetwo}}\\
J_{1,5} &= \decpatternww{scale=0.75}{ 7 }{ 1/3, 2/4, 3/7, 4/2, 5/5, 6/6, 7/1 }{}{2/4,3/7}{0/5,0/6,1/3,1/4,1/5,1/6,1/7,2/5,2/6,2/7,3/7}{}{3/5/4/7/{\onetwo}} \,\,
J_{1,6} = \decpatternww{scale=0.75}{ 8 }{ 1/8, 2/3, 3/4, 4/7, 5/2, 6/5, 7/6, 8/1 }{}{1/8,3/4,4/7}{1/5,1/6,1/7,1/8,2/3,2/4,2/5,2/6,2/7,2/8,3/5,3/6,3/7,3/8,4/7,4/8}{}{4/5/5/7/{\onetwo}}\,\,
J_{1,7} = \decpatternww{scale=0.75}{ 6 }{ 1/3, 2/5, 3/2, 4/4, 5/6, 6/1 }{}{2/5}{0/4,1/3,1/4,1/5,1/6,2/5,2/6}{}{2/4/3/5/{\onetwo}} \,\,
J_{1,8} = \decpatternww{scale=0.75}{ 7 }{ 1/6, 2/3, 3/5, 4/2, 5/4, 6/7, 7/1 }{}{1/6,3/5}{1/4,1/5,1/6,1/7,2/3,2/4,2/5,2/6,2/7,3/5,3/6,3/7}{}{3/4/4/5/{\onetwo}} \\
J_{1,9} &= \decpatternww{scale=0.75}{ 7 }{ 1/7, 2/3, 3/5, 4/2, 5/4, 6/6, 7/1 }{}{1/7,3/5}{1/4,1/5,1/6,1/7,2/3,2/4,2/5,2/6,2/7,3/5,3/6,3/7}{}{3/4/4/5/{\onetwo}} \,\,
J_{1,10} = \decpatternww{scale=0.75}{ 6 }{ 1/3, 2/6, 3/2, 4/4, 5/5, 6/1 }{}{2/6}{0/4,0/5,1/3,1/4,1/5,1/6,2/6}{}{2/4/3/6/{\onetwo}} \,\,
J_{1,11} = \decpatternww{scale=0.75}{ 7 }{ 1/7, 2/3, 3/6, 4/2, 5/4, 6/5, 7/1 }{}{1/7,3/6}{1/4,1/5,1/6,1/7,2/3,2/4,2/5,2/6,2/7,3/6,3/7}{}{3/4/4/6/{\onetwo}}
\end{align*}
\end{proposition}
\begin{proof}
The proof of this proposition is similar to the proof of Proposition \ref{prop:J2} and therefore omitted.
\end{proof}

Taken together, Lemmas \ref{lem:I's} and \ref{lem:j_3}, with Propositions \ref{prop:J2} and \ref{prop:J1} produce a list of $29$ patterns describing permutations that are not West-$3$-stack-sortable. We can simplify this list considerably by observing that the patterns $J_{1,1}, \dotsc, J_{1,6}$ all imply containment of $I_1$ so we can remove them. Further simplifications of this sort can be made. Also, by considering what happens with the decorated patterns when they contain a certain number of elements in the decorated region the list can be simplified even further. See \cite[Theorem~4.6]{U11W} for the details.

\begin{theorem} \label{thm:W3s}
A permutation is West-$3$-stack-sortable if and only if it avoids the decorated patterns
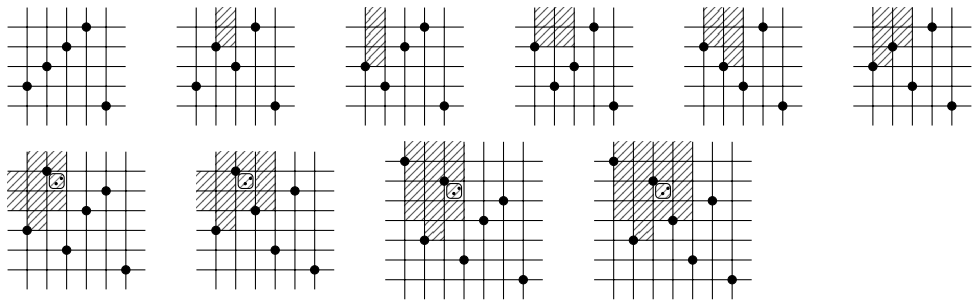
\begin{align*}
&\pattern{scale=0.75}{ 5 }{ 1/2, 2/3, 3/4, 4/5, 5/1 }{} \quad
\pattern{scale=0.75}{ 5 }{ 1/2, 2/4, 3/3, 4/5, 5/1 }{2/4,2/5} \quad
\pattern{scale=0.75}{ 5 }{ 1/3, 2/2, 3/4, 4/5, 5/1 }{1/3,1/4,1/5} \quad
\pattern{scale=0.75}{ 5 }{ 1/4, 2/2, 3/3, 4/5, 5/1 }{1/4,1/5,2/4,2/5} \quad
\pattern{scale=0.75}{ 5 }{ 1/4, 2/3, 3/2, 4/5, 5/1 }{1/4,1/5,2/3,2/4,2/5} \quad
\pattern{scale=0.75}{ 5 }{ 1/3, 2/4, 3/2, 4/5, 5/1 }{1/3,1/4,1/5,2/4,2/5} \\
&\decpatternww{scale=0.75}{ 6 }{ 1/3, 2/6, 3/2, 4/4, 5/5, 6/1 }{}{}{0/4,0/5,1/3,1/4,1/5,1/6,2/4,2/6}{}{2/5/3/6/{\onetwo}} \quad
\decpatternww{scale=0.75}{ 6 }{ 1/3, 2/6, 3/4, 4/2, 5/5, 6/1 }{}{}{0/4,0/5,1/3,1/4,1/5,1/6,2/4,2/6,3/4,3/5,3/6}{}{2/5/3/6/{\onetwo}} \quad
\decpatternww{scale=0.75}{ 7 }{ 1/7, 2/3, 3/6, 4/2, 5/4, 6/5, 7/1 }{}{}{1/4,1/5,1/6,1/7,2/3,2/4,2/5,2/6,2/7,3/4,3/6,3/7}{}{3/5/4/6/{\onetwo}} \quad
\decpatternww{scale=0.75}{ 7 }{ 1/7, 2/3, 3/6, 4/4, 5/2, 6/5, 7/1 }{}{}{1/4,1/5,1/6,1/7,2/3,2/4,2/5,2/6,2/7,3/4,3/6,3/7,4/4,4/5,4/6,4/7}{}{3/5/4/6/{\onetwo}}
\end{align*}
\end{theorem}

\noindent
Note that each of the decorated patterns in the theorem is equivalent to an infinite family of mesh patterns.
\section{An algorithm}
\label{sec:algo}

We shall now automate proving statements such as Theorem~\ref{thm:West}. More precisely we shall provide an algorithm that, given a classical pattern $p$, produces a finite list of (marked) mesh patterns $P$ such that
\[
S^{-1}(\Av(p)) = \Av(P).
\]
The algorithm can be modified for the bubble-sort operator to prove statements such as Proposition~\ref{prop:bubble1243}.
Given the classical pattern $p$ we identify what orderings of the letters in $p$ are possible prior to sorting, and produce a list, denoted $\unSvar(p)$, of candidates which themselves are classical patterns.
\begin{proposition} \label{prop:unSvar}
     Let $p$ be a permutation of a finite set of integers, and let
     the largest letter of $p$ be $m=\max(p)$. Write $p=\alpha
     m\beta$ and $\alpha=a_1a_2\dotsm a_i$. Then
     \[
     \unSvar(p) \,=\, \bigcup_{j=0}^i \left\{\,
         \gamma m\delta \,\colon
         \gamma\in\unSvar(a_1a_2\dotsm a_j),\, \delta\in\unSvar(a_{j+1}\dotsm a_i\beta) 
       \,\right\}
     \]
     contains all classical patterns that can become $p$ after one pass of stack-sort.
\end{proposition}
This proposition gives a recursive algorithm for computing $\unSvar(p)$. Recall that an \emph{inversion} in a permutation is an occurrence of the classical pattern $21$, while a \emph{non-inversion}
is an occurrence of $12$. Note that if two elements in $p$ are part of an inversion, they must also be part of an inversion in all patterns in $\unSvar(p)$. Non-inversion in $p$ place no restrictions on the patterns in $\unSvar(p)$.

\begin{proof}
The idea behind the proof is that letters can only be moved to the left and small letters are stopped by larger letters. This implies that after the largest letter, $m$, has been moved to a particular location we can recurse on what is remaining to the left and the right.
\end{proof}

Note that $\unSvar(132) = \{321,312,132\}$. However, it is easy to check that there is no way that $321$ can become $132$ after sorting. Candidates like this one are removed in lines $9$--$11$ in Algorithm~\ref{algo:ShadeAndMark} below.

\begin{example} \label{ex:3241}
For the pattern $p = 3241$, considered in Lemma~\ref{lem:W_2woShading}, we have $m = 4$, $\alpha = 32$, $\beta = 1$ and $i = 2$. When $j = 0$, $\gamma$ is the empty word and we get the set $\{ 4 \delta \colon \delta \in \unSvar(321) \}$. It is easy to verify that $\unSvar(321) = \{321\}$, so we get the pattern $4321$. When $j = 1$ we get the set $\{ \gamma 4 \delta \colon \gamma \in \unSvar(3),\, \delta \in \unSvar(21) \}$. Again it is easy to check that $\unSvar(21) = \{21\}$ so we get the pattern $3421$. Finally, when $i = 2$ we get the set $\{ \gamma 4 \delta \colon \gamma \in \unSvar(32),\, \delta \in \unSvar(1) \}$ which gives us the pattern $3241$. In total we have the patterns $4321$,$3421$, and $3241$, which are the underlying classical patterns in Lemma~\ref{lem:W_2woShading}.
\end{example}

Before we state the algorithm note that
\begin{enumerate}
\item if $u > v$ is an inversion in $p'$ that becomes a
  non-inversion in $p=S(p')$, then $u$ must stay on the
  stack until $v$ arrives, and therefore we must shade all
  the boxes above $u$ and between $u$ and $v$. Thus there can
  be no elements of $p'$ in this shaded region. This is handled by line 3 in Algorithm~\ref{algo:ShadeAndMark};
\item if $u > v$ is an inversion in $p'$ that becomes an inversion
  in $p = S(p')$, then there must be another element $c > u$ that pops
  $u$ before $v$ is pushed onto the stack, thus maintaining
  the inversion. If such an element is present in $p'$ we need not do
  anything. If there is no such element we need to mark the region
  above $u$ and between $u$ and $v$ with a ``$1$''. This is handled by lines $4$--$16$ in Algorithm~\ref{algo:ShadeAndMark}.
\end{enumerate}

\renewcommand{\algorithmicrequire}{\textbf{Input:}} 
\begin{algorithm}[h]
  \caption{$\mathsf{ShadeAndMark}$}
  \label{algo:ShadeAndMark}
  \begin{algorithmic}[1]\smallskip
    \REQUIRE $\lambda \in \unSvar(p)$\medskip
    \STATE $n := |p|$
    \STATE $\textrm{marks} := \emptyset$
    \STATE $\textrm{shades} := \bigcup \left\{\, \dbrac{\lambda^{-1}(v),\lambda^{-1}(u)-1} \times \dbrac{v, n} \,:\, (u,v) \in \ninv(p)\,\right\} $\medskip
    \FOR {$(u,v) \in \inv(p)$}\smallskip
      \STATE $i := \lambda^{-1}(u)$
      \STATE $j := \lambda^{-1}(v)$\smallskip
    
      \IF {$\lambda(\ell) < u$ for all $\ell \in \dbrac{i+1,j}$}\smallskip
        \STATE $M := \dbrac{i,j-1} \times \dbrac{u,n}  \,\setminus\, \textrm{shades}$\smallskip
        
        \IF {$M=\emptyset$}
          \RETURN This candidate cannot be marked properly        
        \ENDIF\smallskip
        
        \IF {$m \not\subseteq M$ for all $m\in \textrm{marks}$}
          \STATE add $M$ to marks and remove all supersets of $M$
        \ENDIF\smallskip
        
      \ENDIF\smallskip
      
    \ENDFOR\medskip
 
    \RETURN $(\lambda, \textrm{shades}, \textrm{marks})$
  \end{algorithmic}
\end{algorithm}

\noindent
Here $\ninv(p)$ is the set of non-inversions in $p$, and $\inv(p)$ is the set of inversions in $p$.

\begin{example}
For the pattern $p = 3241$ in Lemma~\ref{lem:W_2woShading} we have $\ninv(p) = \{(3,4),(2,4)\}$ and $\inv(p) = \{(3,2),(3,1),(2,1),(4,1)\}$. We saw in Example \ref{ex:3241} that there are three candidates. We only consider $\lambda = 4321$, for which $\lambda^{-1} = \lambda$. The shading is the union of the sets $\{1\}\times\{4\}$ and $\dbrac{1,2}\times\{4\}$ which is $\{(1,4),(2,4)\}$. The for-loop in the algorithm proceeds as follows.
\begin{itemize}
\item $(u,v) = (3,2)$: $M = \{(2,3)\}$ is added to \textrm{marks},
\item $(u,v) = (3,1)$: $M = \{(2,3),(3,3),(3,4)\}$ is a superset of $\{(2,3)\}$, so it is not added to \textrm{marks},
\item $(u,v) = (2,1)$: $M = \{(3,2),(3,3),(3,4)\}$ is added to \textrm{marks},
\item $(u,v) = (4,1)$: $M = \{(3,4)\}$ is added to \textrm{marks} and the superset $\{(3,2),(3,3),(3,4)\}$ is removed.
\end{itemize}
This leaves us with the marking $\{\{(2,3)\},\{(3,4)\}\}$ which is
consistent with Lemma~\ref{lem:W_2woShading}.
\end{example}

The full algorithm calls $\mathsf{ShadeAndMark}$ for all
classical patterns in $\unSvar(p)$. Below is the result
of applying this algorithm to all classical patterns of length 3.
\begin{align*}
  S^{-1}(\Av(123)) &= \Av\left(
    \pattern{scale=0.62}{3}{1/1,2/2,3/3}{},
    \pattern{scale=0.62}{3}{1/1,2/3,3/2}{2/3},
    \pattern{scale=0.62}{3}{1/2,2/1,3/3}{1/2,1/3},
    \pattern{scale=0.62}{3}{1/3,2/1,3/2}{1/3,2/3},
    \pattern{scale=0.62}{3}{1/3,2/2,3/1}{1/3,2/2,2/3}
  \right)\\[1ex]
  S^{-1}(\Av(132)) &= \Av\left(
    \patternsbm{scale=0.66}{3}{1/1,2/3,3/2}{}{2/3/3/4/\scriptscriptstyle{1}}, 
    \patternsbm{scale=0.66}{3}{1/3,2/1,3/2}{1/3}{2/3/3/4/\scriptscriptstyle{1}}\right)
    = \Av\left(\pattern{scale=0.62}{4}{1/1,2/3,3/4,4/2}{},\pattern{scale=0.62}{4}{1/3,2/1,3/4,4/2}{1/3,1/4}\right)\smallskip\\[1ex]
  S^{-1}(\Av(213)) &= \Av\left(
    \patternsbm{scale=0.66}{3}{1/2,2/1,3/3}{}{1/2/2/4/\scriptscriptstyle{1}},
    \pattern{scale=0.66}{3}{1/2,2/3,3/1}{2/3},
    \patternsbm{scale=0.66}{3}{1/3,2/2,3/1}{1/3,2/3}{2/2/3/3/\scriptscriptstyle{1}}
  \right)
  = \Av\left(\pattern{scale=0.62}{4}{1/2,2/3,3/1,4/4}{}, \pattern{scale=0.62}{4}{1/2,2/4,3/1,4/3}{},
             \pattern{scale=0.62}{3}{1/2,2/3,3/1}{2/3},  \pattern{scale=0.62}{4}{1/4,2/2,3/3,4/1}{1/4,2/4,3/4}\right)\\[1ex]
  S^{-1}(\Av(231)) &= \text{See Theorem~\ref{thm:West}}\\[1ex]
  S^{-1}(\Av(312)) &= \Av\left(
    \patternsbm{scale=0.66}{3}{1/3,2/1,3/2}{}{1/3/2/4/\scriptscriptstyle{1}},
    \patternsbm{scale=0.66}{3}{1/3,2/2,3/1}{2/2,2/3}{1/3/2/4/\scriptscriptstyle{1}}\right)
  = \Av\left(\pattern{scale=0.62}{4}{1/3,2/4,3/1,4/2}{},\pattern{scale=0.62}{4}{1/3,2/4,3/2,4/1}{3/2,3/3,3/4}\right)\\[1ex]
  S^{-1}(\Av(321)) &= \Av\left(
    \patternsbm{scale=0.66}{3}{1/3,2/2,3/1}{}{1/3/2/4/\scriptscriptstyle{1}, 2/2/3/4/\scriptscriptstyle{1}}
  \right)
  = \Av\left(\pattern{scale=0.62}{5}{1/4,2/5,3/2,4/3,5/1}{}, \pattern{scale=0.62}{5}{1/3,2/5,3/2,4/4,5/1}{},
             \pattern{scale=0.62}{5}{1/3,2/4,3/2,4/5,5/1}{}\right)
\end{align*}
Bouvel and Guibert~\cite{BG11} have found a bijection between
$S^{-1}(\Av(312))$ and the set of Baxter permutations; they have also
found a bijection between $S^{-1}(\Av(231))$ and $S^{-1}(\Av(132))$.

We note that the algorithm can easily be extended to accept a finite
list of classical patterns.  We also note that a slight modification
of algorithm $\mathsf{ShadeAndMark}$ has been shown to work with the
bubble-sort operator~\cite{CU11}; this algorithm can for example be
used to prove Proposition~\ref{prop:bubble1243}.

\section{Open problems}
\label{sec:op}

The algorithm above describes the preimage of any set $\Av(p)$ where
$p$ is a classical pattern. Can the algorithm be extended to cover the case where
$p$ is a mesh pattern, or even a decorated pattern? Solving this problem
would automate the description of West-$3$-stack-sortable permutations. More generally,
is there a pattern definition that is stable under $S^{-1}$?

West~\cite{W90} conjectured, and Zeilberger~\cite{Z92} proved, that the number of
West-$2$-stack-sortable permutations is given by \mbox{$2(3n)!/((n+1)!(2n+1)!)$}.
Later Dulucq, Gire and West~\cite{DGW96}
found these permutations to be in bijection with rooted non-separable
planar maps. The enumeration of West-$3$-stack-sortable
permutations is completely open, but knowing the
patterns in Theorem~\ref{thm:W3s} could provide some insight.

\section{Acknowledgements}
\label{sec:ack}
We were supported by grant no.\ 090038013 from the Icelandic
Research Fund. We would like to thank the anonymous referees for detailed and constructive comments.
The first author also wishes the express his gratitude to Michael Albert,
Mike Atkinson, Mathilde Bouvel and Mark Dukes for many interesting and
valuable discussions on the topic of sorting operators.

\bibliographystyle{alpha}
\bibliography{Sort_and_preim}
\label{sec:biblio}

\end{document}